\documentclass[11pt, letterpaper, oneside]{amsart}

\usepackage{latexsym, amsmath, amssymb, amsfonts, amscd,bm,stmaryrd}
\usepackage{amsthm}
\usepackage{t1enc}
\usepackage[mathscr]{eucal}
\usepackage{graphicx, pb-diagram}
\usepackage{fancyhdr}
\usepackage{fancybox}
\usepackage{enumerate}
\usepackage{color}
\usepackage[all]{xy}
\usepackage{hyperref}
\usepackage{url}
\usepackage{tikz}

\theoremstyle{plain}
\newtheorem{thm}{Theorem}[section]
\newtheorem*{thmnonumber}{Theorem}

\newtheorem{prop}[thm]{Proposition}
\newtheorem*{propnonumber}{Proposition}
\newtheorem{lemma}[thm]{Lemma}
\newtheorem{cor}[thm]{Corollary}

\theoremstyle{definition}
\newtheorem{defi}[thm]{Definition}

\newtheorem{ex}[thm]{Example}

\theoremstyle{remark} 
\newtheorem{rem}[thm]{Remark}

\newcommand{\ZZ}{\ensuremath{\mathbb Z}}

\newcommand{\RR}{\ensuremath{\mathbb R}}
\newcommand{\g}{\ensuremath{\mathfrak{g}}}

\newcommand{\h}{\ensuremath{\mathfrak{h}}}

\newcommand{\nrt}[1]{\textcolor{black}{#1}}
\newcommand{\nbt}[1]{\textcolor{black}{#1}}

\newcommand{\nft}[1]{\textcolor{black}{#1}}

\definecolor{forest}{rgb}{0,0.5,0}

\begin{document}

\title[Lie $2$-algebra moment maps]{Lie $2$-algebra moment maps\\ in multisymplectic geometry}

\author{Leyli Mammadova}
\email{leyli.mammadova@kuleuven.be}
  
 \author{Marco Zambon}
\email{marco.zambon@kuleuven.be}
\address{KU Leuven, Department of Mathematics, Celestijnenlaan 200B box 2400, BE-3001 Leuven, Belgium.}

 

\begin{abstract}
Consider a closed non-degenerate 3-form $\omega$ with an infinitesimal action of a Lie algebra $\g$. Motivated by the fact that the observables associated to $\omega$ form a Lie 2-algebra, we introduce homotopy moment maps defined on a Lie 2-algebra rather than just on the Lie algebra $\g$.

\nft{We formulate existence criteria and provide a construction for such homotopy moment maps, by characterizing them in terms of cohomology.} 
 \end{abstract}

\maketitle
\date{today}

\setcounter{tocdepth}{1} 
\tableofcontents

\section*{Introduction}   

In symplectic geometry, moment maps play an important role, leading to celebrated theorems: the Marsden-Weinstein-Meyer symplectic reduction, the Atiyah-Guillemin-Sternberg convexity theorem, and the classification of toric symplectic manifolds via Delzant polytopes.
Given an action of a Lie group $G$ on a symplectic manifold $M$, a moment map can be equivalently described as a Lie algebra morphism $\g\ \to (C^{\infty}(M),\{\;,\;\})$ realizing the generators of the action as Hamiltonian vector fields, where  $\{\;,\;\}$ is the Poisson bracket on functions that encodes the symplectic structure on $M$.

Here we consider 2-plectic forms, i.e. closed non-degenerate \emph{3-forms}.
In that case $C^{\infty}(M)$  no longer carries a Poisson bracket. \nft{However, it can be enlarged to a
Lie 2-algebra (a simple kind of $L_{\infty}$-algebra)  canonically attached to $\omega$, which we
denote by $L_{\infty}(M,\omega)$}.
A $\g$ moment map is then defined as an 
$L_{\infty}$-algebra morphism $\g\to L_{\infty}(M,\omega)$ compatible with the action, see  \nrt{\cite[Prop. 5.1]{FRZ}}.
There it is shown (in the wider setting of multisymplectic forms of arbitrary degree) that examples abound, and a \nrt{link with equivariant cohomology is established}. 

In this note we go one step further, replacing the Lie algebra $\g$ by a \emph{Lie 2-algebra} $L$ having $\g$ as its degree $0$ component. There are two main motivations for this:
\begin{itemize}
\item As a $\g$ moment map is an 
$L_{\infty}$-algebra morphism, from an algebraic point of view  it is natural to let the domain be an $L_{\infty}$-algebra rather than just a Lie algebra.
\item While an action by Hamiltonian vector fields might not admit a $\g$ moment map, it  always admits a moment map for a specific Lie 2-algebra {having $\g$ in degree zero}, provided that $H^1(M)=0$  (\cite[Prop. 9.10]{FRZ}, which we recall in Prop. \ref{prop:chrisrg}). {Further, even when $H^1(M)\neq 0$, it admits  a moment map for some Lie 2-algebra (see Prop. \ref{prop:gH}).}
\end{itemize}

\bigskip
\paragraph{\bf Main results}
 Let $(M,\omega)$ be a 2-plectic   manifold, and let $\g\to \mathfrak{X}(M),
x\mapsto v_{x}$ be a Lie algebra morphism taking values in Hamiltonian vector fields.
 Let $L$ be a Lie 2-algebra with vanishing unary bracket, whose degree $0$ component is the Lie algebra $\g$.

 \begin{itemize}
\item 
In \S\ref{sec:cohochar} we show that $L$ moment maps  are in bijection with the primitives of a certain 3-cocycle $\widetilde{\omega}$ 
(constructed out of the $\g$ action on $(M,\omega)$)
in the total complex $CE(L)\otimes \Omega(M)$.
Here $CE(L)$ denotes the Chevalley-Eilenberg complex  of the Lie 2-algebra $L$.

\item
While  the (very large) complex $CE(L)\otimes \Omega(M)$ captures all the   information about $L$ moment maps, it turns out that the complex $CE(L)$ itself captures much of this information, as we show in \S\ref{section:exuniq}.
The $\g$ action on $(M,\omega)$ defines a  3-cocycle $\omega_{3p}\in \wedge^3 \g^*$ in the Chevalley-Eilenberg \nrt{complex} of the Lie algebra $\g$ \nrt{\cite[$\S 9$]{FRZ}}. {Since the projection $L\to \g$ is an $L_{\infty}$-morphism, $\omega_{3p}$} can be regarded as a 3-cocycle in $CE(L)$.
We summarize Thm. \ref{cor:explicit} and Prop. \ref{prop:inverse}:

\begin{thmnonumber}

i) A necessary condition for the existence of $L$ moment maps is that $[\omega_{3p}]_{CE(L)}=0$.

ii) Assume $H^1(M)=0$ and $[\omega_{3p}]_{CE(L)}=0$. Then out of every primitive $\eta\in CE(L)^2$ of $\omega_{3p}$ one can construct a $L$ moment map $\phi^{\eta}$.
This construction recovers all $L$ moment maps, up to inner equivalence.
\end{thmnonumber}
Notice that the above is not only an existence statement, but provides a constructive way to obtain $L$ moment maps.

\item The  condition $[\omega_{3p}]_{CE(L)}=0$ is not very explicit, because it is expressed in terms of the (quite large) complex ${CE(L)}$.
In \S\ref{sec:revisit} we express this condition in terms of the familiar Lie algebra cohomology of $\g$, see Prop. \ref{prop:image}.  As a consequence we can provide criteria -- which are easy to check in practice -- for the existence or non-existence of   $L$ moment maps. They are expressed in terms of a certain 3-cocycle $c_{red}$ of the Lie algebra $\g$ with values in a trivial representation, induced by the trinary bracket of the Lie 2-algebra $L$. We summarize \nft{these criteria} as follows (Prop. \ref{prop:cred0} and Prop. \ref{prop:1dim}):
\begin{propnonumber} 
Assume $[\omega_{3p}]_{\g}\neq 0$.

i) If $[c_{red}]_{\g}=0$, then there exists no $L$ moment map.

ii) Assume $H^1(M)=0$. If  $H^{3}(\g)$ is one-dimensional  and {$[c_{red}]_{\g} \neq 0$}, then there exists a $L$ moment map . 
\end{propnonumber}
We also give an {alternative} characterization of the condition $[\omega_{3p}]_{CE(L)}  =0$ in Prop. \ref{prop:which2}


\end{itemize}

In this diagram we summarize  the relation between the cohomologies (\nft{and the relevant classes}) of the three complexes that appear in \S\ref{sec:cohochar}, \S\ref{section:exuniq}, \S\ref{sec:revisit} respectively:
\begin{align*}
H({CE(L)}\otimes \Omega(M))  \overset{r}{\longrightarrow}  & \;\;H({CE(L)})\;\;  \longleftarrow  H(\g)\\
[\widetilde{\omega}]  \hspace{8mm} \mapsto \;\;&\;\;[\omega_{3p}]_{CE(L)}  \;\;\;\mapsfrom  \;[\omega_{3p}]_{\g}  
\end{align*} 

 \bigskip
\paragraph{\bf Generalizations:}
All the results \nft{of this note also hold for (possibly degenerate) closed 3-forms}, with simple modifications. 

Further, we expect similar results to hold when $\omega$ is \nrt{an} $n$-plectic form and $L$ is a Lie $n$-algebra, for arbitrary $n\ge 2$. For the results of 
\S\ref{sec:cohochar} we expect this because in the \nft{defining conditions for}
$L_{\infty}$-morphisms  from a Lie $n$-algebra to  $L_{\infty}(M,\omega)$, all the  terms \nft{that are quadratic and higher (in the morphism)} can be expressed using the prescribed Lie algebra action\footnote{For this, one uses that $L_{\infty}(M,\omega)$ satisfies property (P) of \cite[\S3.2]{FRZ}.}.
\nft{For the results of \S\ref{section:exuniq}, we expect this} because of \cite[Prop. 9.10]{FRZ}.  
However, for arbitrary $n$, we do not expect that existence criteria can be phrased in terms of Lie algebra cohomology as explicitly as in \S\ref{sec:revisit}.

\bigskip
\paragraph{\bf Conventions:}
Given a graded vector space $V=\oplus_{i\in \ZZ}V^i$ and an integer $k$, we denote by $V[k]$ the graded vector space obtained from $V$ by shifting the degrees by $k$. Explicitly, its degree $i$ component is $(V[k])^i=V^{i+k}$.
\bigskip

\paragraph{\bf Acknowledgements:}
M.Z. thanks Domenico Fiorenza for useful conversations. We thank the referee for her/his comments and insights. We acknowledge partial support by the long term structural funding -- Methusalem grant of the Flemish Government, the FWO under EOS project G0H4518N, the FWO research project G083118N (Belgium).

\section{Lie algebra actions on multisymplectic manifolds}\label{sec:multisympl}
We begin by recalling the relevant notions from multisymplectic geometry. 
{Throughout this note, $M$ denotes a   connected manifold.}

\subsection{Multisymplectic manifolds}

\begin{defi}
A pair $(M,\omega)$ is an \textit{n}\emph{-plectic} manifold, if  $\omega$ is a closed nondegenerate $n+1$ form, i.e., \[d_{dR}\omega=0\] 

\noindent and the map $i{\_}\omega:TM\to \wedge^{n}T^{\ast}M, v\mapsto i_{v}\omega$
is injective.\end{defi}
\begin{defi}
An $(n-1)$-form $\alpha$ on an $n$-plectic manifold $(M,\omega)$ is \emph{Hamiltonian} iff there exists a vector field $v_{\alpha}\in\mathfrak{X}(M)$ such that \[d_{dR}\alpha=-i_{v_{\alpha}}\omega.\]
The vector field $v_{\alpha}$ is the \emph{Hamiltonian vector field} corresponding to $\alpha$.
\end{defi}

\begin{defi}
A \emph{Lie $n$-algebra} is an $L_{\infty}$-algebra concentrated in degrees $0,-1,...,1-n$.
\end{defi}
For definition and properties of $L_{\infty}$-algebras see \cite{lada1993introduction}.
\nbt{For the class of $L_{\infty}$-algebras used in this note, we will spell out the structure in \S \ref{sec:momap2}.}

\vspace {0.2cm} An $n$-plectic manifold is canonically equipped with a Lie $n$-algebra $L_{\infty}(M,\omega)$ defined in the following way \cite[Thm. 5.2]{RogersL}:
\begin{defi}
Given an $n$-plectic manifold, there is a corresponding Lie $n$-algebra $L_{\infty}(M,\omega)=(L,\{[\;,\;...\;,\; ]'_k\})$ with underlying graded vector space 
\[
L^i=\begin{cases}
\Omega_{\textrm{Ham}}^{n-1}(M)\hspace{0.5cm}i=0\\
\Omega^{n-1+i}(M)\hspace{0.2cm}1-n\leq i<0
\end{cases}
\]
and maps $\{[\;,\;...\;,\; ]'_k:L^{\otimes k}\to L\;|\;1\leq k<\infty\} $ defined as 
\[[\alpha]'_1=d_{dR}\alpha,\hspace{0,2cm}\textrm{if }|\alpha|<0\]
and, for $k>1$, \[
[\alpha_{1},\;...\;,\alpha_{k} ]'_k=\begin{cases}
\zeta(k)i(v_{\alpha_{1}}\wedge...\wedge v_{\alpha_{k}})\omega \hspace{0.2cm}\textrm{if }|\alpha_{1}\otimes...\otimes\alpha_{k}|=0\\
0\hspace{3.7cm}\textrm{if }|\alpha_{1}\otimes...\otimes\alpha_{k}|<0,
\end{cases}
\]
where $v_{{\alpha}_i}$ is any Hamiltonian vector field associated to $\alpha_i\in{\Omega}_{\textrm{Ham}}^{n-1}(M)$, $\zeta(k)=-(-1)^{\frac{k(k+1)}{2}}$, 
and $i(\dots)$ denotes contraction with a multivector field: $i(v_{\alpha_{1}}\wedge...\wedge v_{\alpha_{k}})\omega=\iota_{v_{\alpha_{k}}}...\iota_{v_{\alpha_{1}}}\omega$. 
\end{defi}

\vspace{0.2cm}
\subsection{Homotopy moment maps for Lie algebra actions}\label{subsec:liealgaction}
Following \cite{FRZ} we recall:
\begin{defi}\label{def:gmomap}
Let {$\g\to \mathfrak{X}(M), {x\mapsto v_x}$} be a Lie algebra action on an $n$-plectic manifold $(M,\omega)$ by Hamiltonian vector fields. 
A \emph{homotopy moment map} for this action \nbt{(or a $\g$ \emph{moment map} for short)} is an $L_{\infty}$-morphism \[(f_k):\g\to L_{\infty}(M,\omega)\] 
such that \begin{equation*}-i_{v_{x}}\omega=d_{dR}(f_1(x))\hspace{0.5cm}\textrm{for all }x\in\g.\end{equation*}
\end{defi}
\begin{rem}\label{Rem:momapeq}  Such a morphism consists of a collection of $n$ graded skew-symmetric maps \[f_k:{\g}^{\otimes k}\to L, \hspace{0.2cm}1\leq k\leq n\] of degree $|f_k|=1-k$, such that \begin{equation*}-i_{v_{x}}\omega=d_{dR}(f_1(x))\end{equation*}
and the following equations hold:
 \begin{multline*}
\sum\limits_{1\leq i<j\leq k}(-1)^{i+j+1}f_{k-1}([x_i,x_j],x_1,...,\widehat{x_i},...,\widehat{x_j},...,x_k)\\=d_{dR}f_k(x_1,...,x_k)+\zeta(k)i(v_{x_1}\wedge...\wedge v_{x_k})\omega
\end{multline*}
for $2\leq k \leq n$ and
\begin{multline*}
\sum\limits_{1\leq i<j\leq n+1}(-1)^{i+j+1}f_{n}([x_i,x_j],x_1,...,\widehat{x_i},...,\widehat{x_j},...,x_{n+1})=\zeta(n+1)i(v_{x_1}\wedge...\wedge v_{x_{n+1}})\omega,
\end{multline*}
where $v_{x_i}$  is the vector field associated to $x_i$ via the $\g$-action.
\end{rem}

{Fixing a point $p\in M$ defines a degree $(n+1)$-cocycle $\omega_{(n+1)p}$ in the Chevalley-Eilenberg complex of the Lie algebra $\g$, by 
\begin{equation*}
\omega_{(n+1)p}(x_1,..,x_{n+1}):=\omega(v_{x_1},...,v_{x_{n+1}})|_p \text{ for all } x_1,..,x_{n+1}\in \g.  
\end{equation*}
\nbt{The cohomology class $[\omega_{(n+1)p}]_{\g}$ is independent of the choice of point $p$.}
Combining  \cite[Prop. 9.5]{FRZ} and \cite[Thm. 9.6]{FRZ} one has:
\begin{prop}\label{prop:9596}
 The existence of a $\g$ moment map implies that $[\omega_{(n+1)p}]_{\g}=0$. The converse holds 
 if   $H^i(M)=0$ for $1\leq i\leq n-1$.
\end{prop}
}   


\section{Lie 2-algebras actions on 2-plectic manifolds}\label{sec:Lie2momap}

\nbt{Since the ``observables'' $L_{\infty}(M,\omega)$ on an $n$-plectic manifold form an $L_{\infty}$-algebra, it is natural to relax the definition of $\g$ moment map by allowing $\g$ to be an $L_{\infty}$-algebra rather than just a Lie algebra. In this section we do this in the simplest case, i.e. for $n=2$.}	

\subsection{Lie 2-algebras and their Chevalley-Eilenberg complex}\label{subsec:lie2} 
{Recall that a  {Lie $2$-algebra} is an $L_{\infty}$-algebra concentrated in degrees $0,-1$.}  
{This means that the underlying graded vector space is of the form $\h[1]\oplus \g$ , where $\h$ and $\g$ are ordinary vector spaces\footnote{Hence $\h[1]$ is a graded vector space concentrated in degree $-1$.}, which we assume to be finite dimensional.
The multibrackets $\delta,[\;,\;],[\;,\;,\;]$ are as follows}
\begin{align*}
\delta&:\h[1]\to\g\\
[\;,\;]&:{\wedge}^{2}\g\to\g\\
\nbt{[\;,\;]}&:\nbt{\g\otimes\h[1]\to\h[1]}\\
[\;,\;,\;]&:{\wedge}^{3}\g\to\h[1]
\end{align*}
{and satisfy {higher Jacobi identities} that are made explicit in \cite[Lemma 4.3.3]{baez2004higher}.}

We now recall the Chevalley-Eilenberg complex of a Lie 2-algebra (see, e.g.,\cite[\S 6]{sati2008lie}). 
 Consider the maps (see \cite{lada1993introduction} or  \cite[\S 2.2]{ryvkin2016observables})
\begin{align*}
l_1&:\h[2]\to\g[1]\hspace{2.05cm}l_1(h)=-\delta(h)\\
l_2&:{S}^{2}(\g[1])\to\g[1]\hspace{1.33cm} l_2(x,y)=[x,y]\\
l_2&:\g[1]\otimes \h[2]\to\h[2]\hspace{0.94cm}l_2(x,h)=[x,h],\hspace{0.1cm} l_2(h,x)=-[h,x]\\
l_3&:{S}^{3}(\g[1])\to\h[2]\hspace{1.31cm} l_3(x,y,z)=-[x,y,z]\\
\end{align*}
where \nbt{$S$ denotes the {graded} symmetric algebra}, $x,y\in \g, h\in \h$.
\nbt{We can extend the $l_i$ to maps
$l_i:S^{\bullet \geq 1}(\h[2]\oplus \g[1])\to S^{\bullet \geq1}(\h[2]\oplus \g[1])$ of degree 1,} by the following formulae: 

\begin{align*}
l_2(x_1x_2...x_n)&=\sum_{\sigma\in Sh(2,n-2)} \epsilon(\sigma,x_1,...,x_n)l_2(x_{\sigma(1)},x_{\sigma(2)})x_{\sigma(3)}... x_{\sigma(n)}
\end{align*}
{and similarly for $l_1$ and $l_3$.}
Here $\epsilon(\sigma,x_1,...,x_n)$ is the Koszul sign of $\sigma$ defined by the equality
$x_1 ... x_n=\epsilon(\sigma,x_1,...,x_n)x_{\sigma(1)} ...  x_{\sigma(n)}$ holding in $S^{\bullet \geq 1}(\h[2]\oplus \g[1])$, \nft{and the sum is over all $(2,n-2)$-unshuffles}.
Note that the $x_i$ are elements of $\h[2]\oplus\g[1]$, and their respective degrees are the ones in $\h[2]\oplus\g[1]$.
By dualization, we obtain the following maps $d_i:S^{\bullet \geq 1}(\h[2]\oplus \g[1])^{\ast}\to S^{\bullet \geq 1}(\h[2]\oplus \g[1])^{\ast}$: {for all $\xi\in  
S^{k}(\h[2]\oplus \g[1])^{\ast}$,}
  \begin{align*}
d_1(\xi)(x_1...x_k)&=\xi(l_1(x_1...x_k))\\
d_2(\xi)(x_1...x_{k+1})&=\xi(l_2(x_1...x_{k+1}))\\
d_3(\xi)(x_1...x_{k+2})&=\xi(l_3(x_1...x_{k+2})).\\
  \end{align*}
We combine the $d_i$ into one map  $\nrt{d_{CE(L)}}:=d_1-d_2+d_3:S^{\bullet \geq 1}(\h[2]\oplus \g[1])^{\ast}\to S^{\bullet \geq 1}(\h[2]\oplus \g[1])^{\ast}$ of degree 1.
The fact that $\nrt{d^2_{CE(L)}}=0$ follows from the {higher} Jacobi identities for $L_{\infty}$-algebras (see \cite{lada1993introduction} or  \cite[Lemma 2.13]{ryvkin2016observables}).

\begin{defi}
The complex $${CE(L)}:=(S^{\bullet \geq 1}(\h[2]\oplus \g[1])^{\ast}, \nrt{d_{CE(L)}})$$ is the
\emph{Chevalley-Eilenberg complex} of the Lie 2-algebra $\h[1]\oplus \g$. 
\end{defi}

\subsection{Homotopy moment maps for Lie 2-algebras}\label{sec:momap2}
From now on, for simplicity, we restrict ourselves to the case of  2-plectic forms. {Until the end of this note we assume the following set-up:}

\begin{center}
\fbox{
\parbox[c]{12.6cm}{\begin{center}
$(M,\omega)$ is a   2-plectic manifold,\\
$\g$ is a Lie algebra,\\
$\g\to \mathfrak{X}(M), {x\mapsto v_x}$ is a Lie algebra morphism, such that the $v_x$ are Hamiltonian vector fields.
\end{center}
}}
 \end{center}

{We extend Def. \ref{def:gmomap}. Let  $\h[1]\oplus \g$ be a Lie 2-algebra, whose binary bracket extends the given Lie algebra structure on $\g$.}
\begin{defi}
 A \emph{homotopy moment map}  for the Lie 2-algebra   $\h[1]\oplus \g$ \nbt{(or $\h[1]\oplus \g$ \emph{moment map} for short)} is an $L_{\infty}$-morphism $(f_1,f_2)$
 from $(\h[1]\oplus \g,\delta,[\;,\;],[\;,\;,\;])$ to $(L_{\infty}(M,\omega),d_{dR},[\;,\;]',[\;,\;,\;]')$ such that {for all $x\in \g$} \[-i_{v_x}\omega=d_{dR}(f_1(x)).\]\end{defi}

\begin{rem}\label{rem:linftymor} Explicitly, this means that the components
\begin{align*}
f_{1}:\g&\to{\Omega^{1}_{\textrm{Ham}}}(M),\\
f_{1}: \h&\to C^{\infty}(M),\\
f_{2}:{\wedge}^{2}\g&\to C^{\infty}(M),
\end{align*}
satisfy the following equations {for all $x,y,z\in \g$ and $h\in\h$} (\cite[Def. 6.2]{RogersPre}, see also \cite[Def. 5.3]{HDirac}): 
\begin{align}
\label{eq:fdiff} d_{dR}\circ f_1&=f_1\circ\delta\\
\nonumber d_{dR}f_2(x,y)&=f_1[x,y]-[f_1(x),f_1(y)]'\\
\nonumber f_2(\delta h,x)&=f_1[h,x],\\
\nonumber f_1[x,y,z]-[f_1(x),f_1(y),f_1(z)]'&=f_2(x,[y,z])-f_2(y,[x,z])+f_2(z,[x,y]).
\end{align}
\end{rem}

{Notice that the image of $\delta\colon \h\to \g$ lies in the kernel of the infinitesimal action $\g\to \mathfrak{X}(M)$.}
 Indeed, from equation {\eqref{eq:fdiff}} it follows that the Hamiltonian 1-form $f_1(\delta(h))$ is exact, implying the vanishing of its   Hamiltonian vector field $v_{f_1(\delta(h))}$, which is the infinitesimal generator of the action corresponding to $\delta(h)$.
 {As a consequence,} if the infinitesimal $\g$-action is {effective (in the sense that the Lie algebra morphism $\g\to \mathfrak{X}(M)$ has trivial kernel)}
then the unary bracket  $\delta$  of $\h[1]\oplus \g$ necessarily  vanishes.

Hence, from now on in this note we assume the following (recall that a $L_{\infty}$-algebra is called \emph{minimal} if it has vanishing unary map):
  \begin{center}
\fbox{
\parbox[c]{12.6cm}{\begin{center}
the Lie 2-algebra $\h[1]\oplus \g$ is minimal.
\end{center}
}}
 \end{center}

\begin{rem}
 Any Lie $n$-algebra is $L_{\infty}$-quasi-isomorphic to a minimal Lie $n$-algebra (see \cite[$\S7$]{doubek2007deformation} up to and including Cor. 7.5). Hence this assumption does not imply any loss of generality. We thank Chris Rogers for pointing this out to us. 
\end{rem}

 \nbt{Such a Lie 2-algebra admits a simple well-known description, that we now recall \nbt{(see \cite[Thm. 55]{baez2004higher} for more details)}.
\begin{lemma}\label{lem:minimal}
A minimal Lie 2-algebra corresponds to the following data: 
  \begin{itemize}
\item a Lie algebra $\g$,
\item  a $\g$-representation $\h$,
\item a 3-cocycle $c$ for this representation.
\end{itemize}
The representation  of $\g$ on $\h$ is given by the binary bracket, and the cocycle for this representation is given by the trinary bracket.   
\end{lemma} 
 }
 \begin{proof}
\nbt{ Let $\h[1]\oplus \g$ be a minimal Lie 2-algebra}.
{The higher Jacobi identities reduce to the following,  for $h \in\h$ and $x,y,z,u\in \g$:}
  
  \begin{align}
 \label{eq:liealg} [[x,y],z]+[[y,z],x]+[[z,x],y] &=0\\
  \label{eq:rep}[[x,y],h]-[x,[y,h]]-[[x,h],z]&=0\\
  \label{eq:cocycle}[x,[y,z,u]]-[y,[x,z,u]]+[z,[x,y,u]]-[u,[x,y,z]]&=  [[x,y],z,u]+[y,[x,z],u]+[y,z,[x,u]]\\ \nonumber &-[x,[y,z],u]-[x,z,[y,u]]+[x,y,[z,u]]\end{align}
  
\nbt{The degree $0$ component $\g$ is a Lie algebra due to \eqref{eq:liealg}.}  
That $\h$ is a representation of $\g$ follows from the fact that $[\g,\h]$ lands in $\h$, and from \eqref{eq:rep}. That the trinary bracket $[\;,\;,\;]=:c:{\wedge}^3\g\rightarrow\h$ is a 3-cocycle in the Chevalley-Eilenberg complex of $\g$ with values in $\h$ follows from \eqref{eq:cocycle}:\begin{align*}
    (d_{\g}c)(x,y,z,u)&=x\cdot c(y,z,u)-y\cdot c(x,z,u)+z\cdot c(x,y,u)-u\cdot c(x,y,z)\\&-c([x,y],z,u)+c([x,z],y,u)-c([x,u],y,z)\\&-c([y,z],x,u)+c([y,u],x,z)-c([z,u],x,y)=0.
       \end{align*}
\end{proof}

\section{A cohomological characterization of Lie 2-algebra moment maps}\label{sec:cohochar}
The main observation in \cite{FLRZ} and  \cite{LeonidTillmannComoment} is that there is a complex that allows to encode efficiently moment maps for Lie algebra actions, showing in particular that the latter form an affine subspace. In this section we obtain an analogous result for Lie 2-algebras.

Let $\h[1]\oplus \g$ be a minimal Lie 2-algebra, let $\omega$ be a 2-plectic form on a manifold $M$, and let $\g\to \mathfrak{X}(M),x\mapsto v_{x}$ be a Lie algebra morphism {taking values in Hamiltonian vector fields}.

\subsection{A cohomological characterization}\label{subsec:homochar}
Let $\nbt{({CE(L)},d_{CE(L)})}$ be the Chevalley-Eilenberg complex  of the  Lie 2-algebra  $\h[1]\oplus \g$, see \S \ref{subsec:lie2}.
Consider the double complex $$C:={CE(L)}\otimes \Omega(M),$$ where $(\Omega(M),d_{dR})$ is the de Rham complex of $M$.
We denote the resulting total complex by $(C,d_{tot})$, where \[d_{tot}:=\nrt{d_{CE(L)}}\otimes1+1\otimes d_{dR}.\]
Define
\begin{equation*} \omega_k:{\wedge}^k\g\to{\Omega}^{3-k}(M), (x_1,...,x_k)\mapsto \nbt{i(v_{x_1}\wedge...\wedge v_{x_k})}\omega \end{equation*}
and
\begin{equation*}\
\widetilde{\omega}:=\sum_{k=1}^{3}{(-1)}^{k-1}\omega_k
\end{equation*}
Note that $\widetilde{\omega}$ is a degree $3$ element of ${CE(L)}\otimes \Omega(M)=C$, \nbt{using the canonical identification $\wedge \g^*\cong S(\g[1])^*$.}
The following lemma is analog to \cite[Cor. 2.4]{FLRZ}.
\begin{lemma}
	 $\widetilde{\omega}$ is $d_{tot}$-closed.
\end{lemma}
\begin{proof}
It was shown in \cite{FLRZ} that $\widetilde{\omega}\in{S}^{\bullet \geq1}({\g[1]})^{\ast}\otimes\Omega(M)$ is closed with respect to the differential $D:=d_{\g}\otimes1+1\otimes d_{dR}$, where $d_{\g}$ is the Chevalley-Eilenberg differential of $\g$.

The inclusion $j:(S^{\bullet \geq1}({\g[1]})^*,d_{\g})\to(S^{\bullet \geq1}{(\h[2]\oplus\g[1])}^*,d_{CE(L)})=({CE(L)},d_{CE(L)})$ is a chain map, i.e., $j(d_{\g}(\xi))={d_{CE(L)}}(j(\xi))$ for all $\xi\in S^{k}(\g[1])^{\ast}$. This follows from:

1) $\nrt{d_{CE(L)}}=-d_{2}+d_{3}$, because the unary bracket   of $\h[1]\oplus\g$ vanishes, 

2) $-d_{2}=d_\g$ on elements of $S^{k}(\g[1])^{\ast}$, 

3) $d_{3}(\xi)=0$, {because the trinary bracket of  $\h[1]\oplus\g$ takes values in $\h[1]$}.

Thus, the inclusion $(S^{\bullet \geq1}({\g[1]})^*\otimes\Omega(M),D)\to({{CE(L)}} \otimes\Omega(M),d_{tot})$ is also a chain map, which means that $\widetilde{\omega}\in {CE(L)}\otimes\Omega(M)$ is also closed with respect to $d_{tot}$.
\end{proof}

Analogously to \cite[Prop. 2.5]{FLRZ}{\cite[Prop. 6.1]{FRZ}} we obtain:
\begin{prop}\label{prop:primitive}
There is a bijection  
\nbt{$$\{\text{moment maps for $\h[1]\oplus\g$}\}\cong \{\mu \in C^2: d_{tot}\mu=\widetilde{\omega}\}.$$
 More precisely,
for all $$\mu=\mu_1|_{\g}+\mu_1|_{\h}+\mu_2\in C^2 ={\g[1]}^*\otimes{\Omega}^1(M)\;\oplus\; ({\h[2]}^*\oplus S^2({\g}[1])^{\ast})\otimes C^{\infty}(M)$$ we have: $d_{tot}\mu=\widetilde{\omega}$
iff}
\begin{align*}
\mu_1|_{\g}:\g &\to{\Omega}_{\textrm{Ham}}^1\\
\mu_1|_{\h}:\h &\to C^{\infty}(M)\\
\mu_2:{\wedge}^2\g &\to C^{\infty}(M)
\end{align*} 
 are the components of a $\h[1]\oplus\g$ moment map.
\end{prop}

{Notice that by Prop. \ref{prop:primitive} the set of moment maps for   $\h[1]\oplus\g$
forms an affine space.}

\begin{figure}[h!]
\begin{center}
	\begin{tikzpicture}
	\node (a)  at (0,0) {$\g[1]^{\ast}\otimes C^{\infty}(M)$};
	\node (b)  at (0,1.5) 
	{\fbox{$\big(\h[2]^*\oplus S^{2}(\g[1])^{\ast}\big)\otimes C^{\infty}(M)$}};
	\node (c)  at (0,3) {$\big(\h[2]^*\otimes\g[1]^*\oplus S^{3}(\g[1])^{\ast}\big)\otimes C^{\infty}(M)$};
	\node (d) at (0,4.5) {...};
	
	\node (e) at (6,0) {\fbox{$\g[1]^{\ast}\otimes \Omega^{1}(M)$}};
	
	\node (f) at (6,1.5) {$\big(\h[2]^*\oplus S^{2}(\g[1])^{\ast}\big)\otimes \Omega^{1}(M)$};
	\node (g) at (6,3) {...};
	\node (h) at (10,0) {$\g[1]^{\ast}\otimes \Omega^{2}(M)$};
	\node (i) at (10,1.5) {...};
	\path[->] (a) edge node[left] {$d_{CE(L)}$} (b);
	\path[->] (b) edge node[left] {$d_{CE(L)}$} (c);
	\path[->] (c) edge node[left] {$d_{CE(L)}$} (d);
	\path[->] (a) edge node[below] {$d_{dR}$} (e);
	\path[->] (b) edge  (f);
	\path[->] (f) edge  (i);
	\path[->] (e) edge node[below] {$d_{dR}$} (h);
	\path[->] (f) edge  (g);
	\path[->] (e) edge  (f);
	\path[->] (c) edge  (g);
	\path[->] (h) edge  (i);
	
	\end{tikzpicture}
	\caption{{The double complex $C={CE(L)}\otimes \Omega(M)$, with 
highlighted  the components   of total degree $2$, i.e. $C^2$.}}
\end{center}
\end{figure}

\begin{proof}
Recall that 
a moment map $\phi=(\phi_1|_{\g},\phi_1|_{\h},\phi_2)$ for  $(\h[1]\oplus\g,[\;,\;],[\;,\;,\;])$   has to satisfy the following equalities for $x,y,z\in\g,h\in\h$, \nft{by Remark \ref{rem:linftymor}.}
\begin{align}
\label{eq:start_one}d_{dR}(\phi_1|_{\g}(x)) &=-i_{v_x}\omega\\
d_{dR}(\phi_1|_{\h}(h)) &=0\\
d_{dR}(\phi_2(x,y)) &=\phi_1|_{\g}[x,y]-[\phi_1|_{\g}(x),\phi_1|_{\g}(y)]'\\
\phi_1|_{\h}[h,x] &=0 \\
\label{eq:finish_one}\phi_1|_{\h}([x,y,z])-[\phi_1|_{\g}(x),\phi_1|_{\g}(y),\phi_1|_{\g}(z)]' &=\phi_2(x,[y,z])-\phi_2(y,[x,z])+\phi_2(z,[x,y]) , 
\end{align}
where $[\;,\;]'$ and $[\;,\;,\;]'$ are the binary and the trinary bracket of $L_{\infty}(M,\omega)$.

{Now let $\mu\in C^2$.}
Comparing the  components of $d_{tot}\mu\in C^3$ and $\widetilde{\omega}\in C^3$ {with the same bi-degree}, we see that the equation $d_{tot}\mu=\widetilde{\omega}$ is equivalent to the following five equations: \begin{align*}
(1\otimes d_{dR})(\mu_1|_{\g}) &=\omega_1 \\
(1\otimes d_{dR})(\mu_1|_{\h}) &=0 \\
-(d_2\otimes 1)(\mu_1|_{\g})+(1\otimes d_{dR})\mu_2 &=-\omega_2 \\
-(d_2\otimes 1)(\mu_1|_{\h}) &=0\\
(d_3\otimes 1)(\mu_1|_{\h})-(d_2\otimes 1)\mu_2 &=\omega_3 
\end{align*}
Rewriting these equations in terms of the respective Lie 2-algebra brackets and evaluating on $x,y,z\in \g, h\in\h$, we get the following equations:
\begin{align}
\label{eq:start_two} d_{dR}(\mu_1|_{\g}(x))&=-i_{v_x}\omega\\
d_{dR}(\mu_1|_{\h}(h))&=0\\
-\mu_1|_{\g}([x,y])+d_{dR}(\mu_{2}(x,y))&=-[\mu_1|_{\g}(x),\mu_1|_{\g}(y)]'\\
\mu_1|_{\h}([h,x])&=0\\
\label{eq:finish_two}
-\mu_1|_{\h}([x,y,z])-\mu_2([x,y],z)+\mu_2([x,z],y)-\mu_2([y,z],x)&=-[\mu_1|_{\g}(x),\mu_1|_{\g}(y),\mu_1|_{\g}(z)]'
\end{align}

Comparing equations \eqref{eq:start_one} - \eqref{eq:finish_one} and \eqref{eq:start_two} - \eqref{eq:finish_two}, {we   see 
that $(\mu_1|_{\g}, \mu_1|_{\h}, \mu_2)$ is a moment map if{f} $d_{tot}\mu=\widetilde{\omega}$.}
\end{proof}

\begin{rem}
{We proved Prop. \ref{prop:primitive} by computations. There is a conceptual approach to this result, using homotopy theory (we thank the referee for pointing this out). The conceptual approach is the same as for the special case of moment maps for $\g$, which was explained in \cite[Rem. 6.2]{FRZ}.}
\end{rem}

\section{\nbt{Existence results and a construction}}\label{section:exuniq}
 We use the characterization of moment maps for Lie 2-algebras given in \S \ref{sec:cohochar} to {obtain existence results and} construct explicit examples.
 
Again, let $\h[1]\oplus \g$ be a minimal Lie 2-algebra, let $\omega$ be a 2-plectic form on a manifold $M$, and let $\g\to \mathfrak{X}(M),
x\mapsto v_{x}$ be a Lie algebra morphism taking values in Hamiltonian vector fields.

\subsection{{Two immediate  existence results}}\label{subsec:imme}

Given a minimal Lie 2-algebra $\h[1]\oplus \g$, the projection $\h[1]\oplus \g\to \g$ is  a strict $L_{\infty}$-morphism,  hence we immediately have:
\begin{cor}\label{cor:gtohg}
\nbt{From a $\g$ moment map, by composition with the above projection, one obtains a $\h[1]\oplus \g$ moment map.}  
\end{cor}

\begin{rem}\label{rem:DGLA}
{The inclusion $\g \to \h[1]\oplus \g$ however is a not a strict $L_{\infty}$-morphism in general: it is if{f} $\h[1]\oplus \g$ is a DGLA. In particular, moment maps for a DGLA $\h[1]\oplus \g$ exist if{f} $\g$ moment maps exist.}
\end{rem}

Now fix $p\in M$. {As recalled in Prop. \ref{prop:9596}, the existence of a $\g$ moment map is equivalent to $[\omega_{3p}]_{\g}=0$, 
 under the mild assumption that   $H^1(M)=0$. {Recall that $\omega_{3p}$ is the Lie algebra 3-cocycle defined at the end of \S\ref{subsec:liealgaction}, i.e.   for all $x,y,z\in \g$:
\begin{equation}
\label{eq:omega3p}
\omega_{3p}(x,y,z)=\omega(v_{x},v_{y},v_{z})|_p.  
\end{equation} 
  } 

When $[\omega_{3p}]_{\g}\neq 0$ there exists no moment map for $\g$, but there always exists a moment map for a slightly larger Lie 2-algebra. Indeed, denote by $$\RR[1]\oplus_{-\omega_{3p}} \g$$
 the Lie 2-algebra underlying $\RR[1]\oplus \g$ with brackets $[x,y]:={[x,y]}_{\g}$ for $x,y\in\g$, trinary bracket equal to $-\omega_{3p}$,  and all other brackets being trivial.
We paraphrase \nbt{a special case of} \cite[Prop. 9.10]{FRZ}:
\begin{prop} \label{prop:chrisrg}
{If $H^1(M)=0$, then there exists a moment map for $\RR[1]\oplus_{-\omega_{3p}} \g$.}
\end{prop}
 {The construction of \cite[Prop. 9.10]{FRZ} is as follows: let $\gamma_1\colon \g\to {\Omega}_{\textrm{Ham}}^{1}(M)$ be a linear map such that $\gamma_1(x)$ is a Hamiltonian 1-form for the Hamiltonian vector field $v_x$, for all $x$.}
{Then a} moment map $\gamma=(\gamma_1, \gamma_2)$ is given by:
\begin{align} 
\label{eq:gamma1}
\gamma_1&\colon\g\to{\Omega}_{\textrm{Ham}}^{1}(M)
\\
\gamma_1&\colon\RR\to C^{\infty}(M),\hspace{0.2cm}r\mapsto r,\hspace{0.2cm} \nonumber\\
\label{eq:gamma2} {\gamma}_2&\colon {\wedge}^2\g\to C^{\infty}(M)
\end{align}
{where, for all $x,y \in \g$, $\gamma_2(x,y)$} is the unique solution of
the equation
\begin{equation*}
 \gamma_1([x_1,x_2])-i({v_{x_1}\wedge v_{x_2})}\omega=d_{dR}({{\gamma}_2}(x_1,x_2))
\end{equation*}
with $\gamma_2(x,y)_p=0$.

\nbt{\emph{From now on we fix a linear map $\gamma_1\colon \g\to {\Omega}_{\textrm{Ham}}^{1}(M)$ as above, and consequently 
a moment map $\gamma$ for $\RR[1]\oplus_{-\omega_{3p}} \g$.}}

\subsection{\nbt{A constructive existence result in terms of $H^3({CE(L)})$.}}
\label{subsec:obstr}
\nbt{By Prop. \ref{prop:primitive}, moment maps for $\h[1]\oplus \g$  exist if{f} the cohomology class of $\widetilde{\omega}\in C^3$ vanishes.}
In this subsection, which is inspired by \cite[\S 5]{FLRZ}, {we obtain existence results for $\h[1]\oplus \g$ moment maps in terms of the cohomology of the {Chevalley-Eilenberg} complex ${CE(L)}$ of this  Lie 2-algebra}. \nbt{Notice that the latter is smaller  than $C$, and thus more manageable.} {Further, we give an explicit construction of $\h[1]\oplus \g$ moment maps.}

Fix a point $p \in M$. The map 
\begin{align*}r:(C,d_{tot})&\to({CE(L)},{d_{CE(L)}})\\ 
\eta\otimes\alpha&\mapsto\eta\cdot\alpha_p,
\end{align*}
is a chain map, where $\alpha_p\in\RR$ is declared to vanish if $\alpha\in{\Omega}^{\geq1}(M)$. Since $\widetilde{\omega}$ is $d_{tot}$-closed by Lemma 3.1, its image $r(\widetilde{\omega})=\omega_{3p}\in {CE(L)}^3$ is ${d_{CE(L)}}$-closed, hence it defines a class $[\omega_{3p}]_{CE(L)}$ in $H^3({CE(L)})$, the Chevalley-Eilenberg cohomology of $\h[1]\oplus\g$. 

\begin{prop} \label{cor:exact}
If a $\h[1]\oplus\g$ moment map   exists, then $[\omega_{3p}]_{CE(L)}=0$.
\end{prop}
\begin{proof}
By Proposition 3.2, if a moment map exists, then $\nbt{[\widetilde{\omega}]_C=0\in H^3(C)}$. Hence we have  $[\omega_{3p}]_{CE(L)}=[r]([\widetilde{\omega}]_C)=0$, where $[r]:H^{3}(C)\to H^{3}({CE(L)})$ is the induced map in cohomology.
\end{proof}

Conversely we are now going to show that, if $[\omega_{3p}]_{CE(L)}=0$ and certain cohomological assumptions on $M$ are satisfied, there exists a moment map \nbt{for $\h[1]\oplus\g$. \nbt{Our approach is constructive.}
}
\begin{rem} \nbt{An alternative approach, which however requires stronger assumptions and is  not constructive, is the following.}
Assume that $H^1(M)=H^2(M)=0$. By the K\"unneth theorem \nft{(see for instance \cite[Thm 3B.5]{Hatcher})},
\nbt{$H^3(C)\cong H^3({CE(L)})\otimes H^0(M)\cong H^3({CE(L)})$,}
and the map  $[r]:H^3(C)\to H^3({CE(L)})$ induced in cohomology is an isomorphism. 
Thus, if $[\omega_{3p}]_{CE(L)}=0$, \nbt{then $[\widetilde{\omega}]_C=0$,}
and by Prop. \ref{prop:primitive} there exists a moment map.
\end{rem}

\nbt{In the following, given an element  ${\eta}$ of
$${CE(L)}^2={\h[2]}^\ast \oplus S^2({\g[1]})^\ast,$$ 
we denote the first and second component  of $\eta$ respectively by ${\eta}|_{\h[2]}$ and ${\eta}|_{S^2({\g[1]})}$.}

\begin{lemma} \label{lemma:star}
An element ${\eta} \in {CE(L)}^2$
satisfies ${d_{CE(L)}}{\eta}=\omega_{3p}$ iff 
\begin{equation}\label{eq:star}
\begin{cases}
d_{\g}({\eta}|_{S^2{\g[1]}})+d_3({\eta}|_{\h[2]})=\omega_{3p} \\ 
d_2({\eta}|_{{\h[2]}})=0
\end{cases}
\end{equation} 
\end{lemma}

\begin{proof}
The claim follows easily by computing ${d_{CE(L)}}\eta$ and comparing the respective components of ${d_{CE(L)}}\eta$ and $\omega_{3p}$ in ${CE(L)}$.
\end{proof}

\begin{rem}\label{rem:gh}
{We will consider this system in detail in \S \ref{subsec:h3g}.
For the moment we only remark that the second equation is equivalent  to ${\eta}|_{{\h[2]}}$ lying in the subspace $[\g,\h]^{\circ}:=\{\xi\in\h^{*}:\xi(v)=0, \hspace{0.1cm}\forall v\in[\g,\h]\}$.}
\end{rem}
 
\begin{lemma}\label{lem:solstarlimor}
There is a bijection  between
\nbt{ $$\{\eta \in {CE(L)}^2: \nrt{d_{CE(L)}}\eta=\omega_{3p}\}$$
 and
$$\{\text{$L_{\infty}$-morphisms $f\colon\h[1]\oplus \g \to {\RR[1]\oplus_{-\omega_{3p}} \g}$ with $f_1|_{\g}=Id_{\g}$}\}.$$ }
{The bijection maps $\eta=\eta|_{\h[2]}+\eta|_{S^2\g[1]}$ to 
the $L_{\infty}$-morphisms $f$ with   components $f_1=(\eta|_{\h[2]},Id_{\g})$ and $f_2=\eta|_{S^2\g[1]}$.}

\end{lemma}
\begin{proof}
We will denote the \nbt{trinary} bracket of ${\RR[1]\oplus_{-\omega_{3p}} \g}$ by \nbt{$[\;,\;,\;]^{\diamond}$}.
For $f$ to be an $L_{\infty}$ morphism $f:\h[1]\oplus\g\to {\RR[1]\oplus_{-\omega_{3p}} \g}$, its  components 
\begin{align*}
f_1|_{\g}:\g&\to\g\\
f_1|_{\h}:\h&\to\RR \\
f_2:{\wedge}^2\g&\to\RR
\end{align*}
 must satisfy the following relations for $x,y,z\in\g,h\in\h$: 
\begin{align}
\label{firstsys} & f_1|_{\h}[h,x]=0,\hspace{0.2cm}\\
& f_1|_{\g}[x,y]-[f_1|_{\g}(x),f_1|_{\g}(y)]_{\g}=0\nonumber\\
 \label{secondsys}& f_1|_{\h}([x,y,z])-[f_1|_{\g}(x),f_1|_{\g}(y),f_1|_{\g}(z)]^{\diamond}=f_2(x,[y,z])-f_2(y,[x,z])+f_2(z,[x,y]). 
\end{align}
These equalities follow from \cite[Def. 5.3]{HDirac}.

\nft{Clearly} \eqref{firstsys} is equivalent to $d_2(f_1|_{\h})=0$. {When $f_1|_{\g}=Id_{\g}$,} rewriting \eqref{secondsys} using the definition of the bracket $[\;,\;,\;]^{\diamond}$, we get: \vspace{-0.2cm}\[\nbt{f_1|_{\h}}([x,y,z])-f_2(x,[y,z])+f_2(y,[x,z])-f_2(z,[x,y])=-\omega_{3p}(x,y,z)\]
The latter equality can be written as
$
d_{\g}(f_2)+d_3(f_1|_{\h}) = \omega_{3p}.
$
{Applying Lemma \ref{lemma:star} concludes the proof}.
\end{proof}

By composition, we immediately obtain the following result, which also provides {an explicit construction for $\h[1]\oplus \g$ moment maps.}
\begin{thm}\label{cor:explicit} Assume $H^1(M)=0$.
 Let $\eta \in {CE(L)}^2$ satisfy $d\eta=\omega_{3p}$. Then $$\nbt{\phi^{\eta}}:=\gamma\circ f$$ is a moment map for $\h[1]\oplus \g$, where $f$
 is \nft{constructed out} of $\eta$ as in Lemma \ref{lem:solstarlimor},
and $\gamma$ is   given {just below Prop. \ref{prop:chrisrg}.}
\end{thm}
\begin{rem} \label{rmk:explicit} 
i) Explicitly, $\nbt{\phi^{\eta}}$
is given \nbt{as follows, where $x,y\in \g, h\in \h$:}
\begin{align*}
\nbt{\phi^{\eta}_1}(x) &= \gamma_1(x)\\ 
\nbt{\phi^{\eta}_1}(h) &= \gamma_1(\eta(h))=\eta(h)\\
\nbt{\phi^{\eta}_2}(x,y) &= \gamma_1(f_2(x,y))+\gamma_2(f_1(x),f_1(y))\\ 
&=\eta(x,y) +\gamma_2(x,y). 
\end{align*}

ii) When $\h=0$, the moment map $\nbt{\phi^{\eta}}$ for $\g$ agrees with the one constructed in \cite[Prop. 9.6]{FRZ}. \nbt{(Notice that when $\h=0$, the choice of $\eta$ amounts to  a choice of primitive of
$\omega_{3p}$ in the Chevalley-Eilenberg complex of $\g$.)}

iii) {The \nbt{ ${\RR[1]\oplus_{-\omega_{3p}} \g}$ moment} map $\gamma$ itself is obtained as in Thm. \ref{cor:explicit} from the solution $\eta(r)=r,\hspace{0.1cm}\eta(x,y)\equiv0,\hspace{0.1cm} r\in\RR,\hspace{0.1cm} x,y\in\g$ of the equation ${d_{CE(L)}}\eta=\omega_{3p}$. Indeed, under the bijection of Lemma \ref{lem:solstarlimor}, $\eta$ corresponds to the identity on 
$\RR[1]\oplus_{-\omega_{3p}} \g$.}
\end{rem}

{
By combining the results of Prop. \ref{cor:exact} and  Thm. \ref{cor:explicit}, we summarize as follows the existence results obtained in this subsection:}
\begin{cor} \label{prop:combined}
Let $(M,\omega)$ be a 2-plectic manifold, and $\g\to \mathfrak{X}(M)$ a Lie algebra taking values in Hamiltonian vector fields. Let $\h[1]\oplus \g$ be a minimal Lie 2-algebra. Fix $p\in M$.

If a moment map for $\h[1]\oplus\g$ exists, then $[\omega_{3p}]_{CE(L)}=0$. {The converse holds whenever $H^1(M)=0$.}
\end{cor}

\subsection{A uniqueness result}\label{subsec:unique} \nbt{In this subsection we assume that   $H^1(M)=0$.}
{In \S \ref{subsec:obstr} we addressed the existence of  moment maps for $\h[1]\oplus\g$.}
Here we show that,  any moment map for $\h[1]\oplus\g$ is cohomologous to one constructed by composition in Thm. \ref{cor:explicit}.
 
 Fix $p\in M$.
Consider again the map introduced in \S \ref{subsec:obstr}:
\begin{align*}r:(C,d_{tot})&\to({CE(L)},{d_{CE(L)}})\\ 
\eta\otimes\alpha&\mapsto\eta\cdot\alpha_p.
\end{align*}
\nbt{We remark that  the map  induced in cohomology in degree 2
\[[r]:H^{2}(C)\to H^{2}({CE(L)})\]
is an isomorphism. This follows from the fact that,
by the K\"unneth theorem,
 \[H^2(C)=H^{2}({CE(L)}\otimes\Omega(M))\cong H^{2}({CE(L)})\otimes \RR 
,\]
where we used that ${CE(L)}$ is concentrated in positive degrees and
  $H^{1}(M)=0$.}

\nbt{To any $\eta\in {CE(L)}^2$ with ${d_{CE(L)}}\eta=\omega_{3p}$, in  Thm. \ref{cor:explicit} we associated a $\h[1]\oplus \g$ moment map $\phi^{\eta}$, which we now view as an element of $C^2$.} 

\begin{lemma}\label{lem:r}
For all $\eta$ as above:	$r(\phi^{\eta})=\eta$
\end{lemma}
\begin{proof}
	Using Remark \ref{rmk:explicit} i) we find: 
	\begin{align*}
	\phi_1^{\eta}|_{\g} &=\gamma_1|_{\g}\hspace{3.3cm}\in{\g[1]}^{\ast}\otimes \Omega^{1}(M)\\
	\phi_1^{\eta}|_{\h}&=\eta|_{\h[2]} \otimes 1 \hspace{2.45cm}\in{\h[2]}^{\ast}\otimes C^{\infty}(M)\\
	\phi^{\eta}_2&=(\eta|_{S^2\g[1]}\otimes 1 + \gamma_{2}) \hspace{1cm}\in {S}^{2}{\g[1]}^{\ast}\otimes C^{\infty}(M)
	\end{align*}
	Now notice that 
	$r(\phi_1^{\eta}|_{\g})=0$ by definition of the map $r$, and
	$r(\phi_1^{\eta}|_{\h})
	=\nbt{\eta|_{\h[2]}}$.
	
	Finally $r(\phi^{\eta}_2)
	=\eta|_{{S}^{2}{\g[1]}}+r(\gamma_{2})=\eta|_{S^2\g[1]}$. Indeed
$r(\gamma_{2})	=\gamma_{2p}=0$,  because $\gamma_2(x,y)$ vanishes at point $p$ for any $x,y\in\g$ by construction (see \eqref{eq:gamma2}).
	Hence
	$r(\phi^{\eta})
	=\eta|_{\h[2]}+\eta|_{S^{2}\g[1]}=\eta$.
\end{proof}

\nbt{The difference of any two moment maps is a closed element of $C^2$, by Prop. \ref{prop:primitive}.}
Extending \cite[Rem. 7.10] {FLRZ} we define:
\begin{defi}
Two moment maps $\mu, \mu'\in \nbt{C^2}$ are called 
\emph{inner equivalent} if $\mu - \mu'=d_{tot}\alpha$ for some $\alpha\in C^{1}$. 
\end{defi}

{The following proposition gives conditions to ensure that all 
	moment maps are equivalent to those constructed earlier.
}
\begin{prop}\label{prop:inverse} Let $M$ be a manifold with $H^1(M)=0$. If $\mu\in C^{2}$ is a moment map for $\h[1]\oplus\g$, then $\mu$ and $\phi^{r(\mu)}$ are inner equivalent.
\end{prop}
\begin{rem} Note that, since $r$ is a chain map and by Prop. \ref{prop:primitive},
if $\mu$ is a moment map, then $r(\mu)$ is a solution of   ${d_{CE(L)}}\eta=\omega_{3p}$. Hence $\phi^{r(\mu)}$ is indeed well-defined.
\end{rem}

\begin{rem} \nft{In \S\ref{section:exuniq} we fixed a choice of linear map $\gamma_1\colon \g\to {\Omega}_{\textrm{Ham}}^{1}(M)$ providing Hamiltonian 1-forms for the generators of the  action. For any  $\eta \in {CE(L)}^2$ satisfying $d\eta=\omega_{3p}$,
the moment map $\phi^{\eta}$ constructed in Thm. \ref{cor:explicit} depends on this choice. However, making a different choice for 
$\gamma_1$ delivers a moment map that is inner equivalent to $\phi^{\eta}$. This can be seen using Lemma \ref{lem:r} and Prop. \ref{prop:inverse}, or also by a direct computation.}
\end{rem}

\begin{proof}
We have
	$r(\mu -\phi^{r(\mu)})=r(\mu)-r(\phi^{r(\mu)})=0$ 	by Lemma \ref{lem:r}. In particular, for the map $[r]$ induced in cohomology, we have $[r][\mu -\phi^{r(\mu)}]=0$. But 
	{the map $[r]$} is an isomorphism \nbt{in degree 2}, hence the cohomology class $[\mu -\phi^{r(\mu)}]$ in $H^{2}(C)$ also vanishes, i.e., $\mu - \phi^{r(\mu)}=d_{tot}\alpha$ for some $\alpha\in C^{1}$.
\end{proof}

Prop. \ref{prop:inverse} immediately implies: 
\begin{cor}\label{cor:universal}
The Lie 2-algebra $\RR[1]\oplus_{-\omega_{3p}}\g$ is \emph{universal} in the following sense: provided $H^1(M)=0$,  any moment map for a Lie 2-algebra $\h[1]\oplus\g$
	is inner equivalent to one that factors through $\RR[1]\oplus_{-\omega_{3p}}\g$. 
\end{cor}

\[\xymatrix{
	\h[1]\oplus \g \ar@{-->}[dr] \ar@{->}[drr] \ar@{->}[ddr]&    &     \\
	& \RR[1]\oplus_{-\omega_{3p}} \g \ar@{->}[d]\ar@{->}_{\gamma}[r]  & L_{\infty}(M,\omega)  \ar@{->}[d]\\
	& \g\ar@{->}[r] & \mathfrak{X}_{\textrm{Ham}}(M)
}
\]

\begin{rem}\label{rem:heis}
{
Let $\rho\colon \g\to \mathfrak{X}(M), x\mapsto v_{x}$ be a Lie algebra morphism taking values in Hamiltonian vector fields. Using  methods from homotopy theory, 
\cite[\S3.4]{FiorenzaRogersSchrLocObs} constructs a Lie 2-algebra $\mathfrak{h}\mathfrak{e}\mathfrak{i}\mathfrak{s}_{\rho}(\g)$ admitting a moment map for this infinitesimal action.
As the referee explained to us, under the assumption $H^1(M)=0$, this Lie 2-algebra is precisely $\RR[1]\oplus_{-\omega_{3p}}\g$, and one can recover Cor. \ref{cor:universal} in a conceptual manner using homotopy theory.
}
\end{rem}


\section{{Revisiting the existence results}}\label{sec:revisit}
As earlier, let $\h[1]\oplus \g$ be a minimal Lie 2-algebra, let $\omega$ be a 2-plectic form on a manifold $M$, and let $\g\to \mathfrak{X}(M),
x\mapsto v_{x}$ be a Lie algebra morphism taking values in Hamiltonian vector fields.

{An answer to the existence question for $\h[1]\oplus \g$ moment maps was given in Cor. \ref{prop:combined}, in terms of the \nbt{cohomology of the Chevalley-Eilenberg complex ${CE(L)}$ of the Lie 2-algebra. However the latter complex is quite large and involved.
In this section we rephrase that answer in two ways}: one that is explicit and easily applicable to examples (\S \ref{subsec:h3g}), and one that is {phrased  directly in terms of the Lie 2-algebra (Prop. \ref{prop:which2}) rather than in terms of its constituents (as in Lemma \ref{lem:minimal})}.}

\subsection{\nbt{An explicit characterization of existence} {in terms of $H^3(\g)$}}
\label{subsec:h3g}  

\nbt{In this subsection, we   answer the question of existence of $\h[1]\oplus \g$ moment maps in terms of the familiar Lie algebra cohomology of $\g$. 
}

{Cor. \ref{prop:combined} expresses the existence of a moment map for $\h[1]\oplus\g$ in terms of the vanishing of $[\omega_{3p}]_{CE(L)}$. Recall that the latter condition is equivalent to the existence of a solution of the system 		$\eqref{eq:star}$.}
We now formulate the first equation of the system \eqref{eq:star} in terms of the Lie algebra cohomology of $\g$.
\begin{lemma}
	For any $\xi \in [\g,\h]^{\circ}\subset \h^*$, the element $d_3\xi\in \wedge^3 \g^*$ is $d_{\g}$-closed.
\end{lemma}
\begin{proof} Recall that $d_3$ was defined in \S\ref{subsec:lie2}. For all $x,y,z,u\in \g$ we compute 
	\begin{align*}
	d_{\g}(d_3\xi)(x,y,z,u) &=-(d_3\xi)([x,y],z,u)+(d_3\xi)([x,z],y,u)-(d_3\xi)([x,u],y,z)\\ 
	& \;\;\;\, -(d_3\xi)([y,z],x,u)+(d_3\xi)([y,u],x,z)-(d_3\xi)([z,u],x,y)\\
	 & = \xi (c([x,y],z,u)) - \xi(c([x,z],y,u)) + \xi(c([x,u],y,z))\\
	 & \;\;\;\, + \xi(c([y,z],x,u)) - \xi(c([y,u],x,z)) + \xi(c([z,u],x,y))\\
	&=  \xi\big(x\cdot c(y,z,u) - y\cdot c(x,z,u) + z\cdot c(x,y,u) - u\cdot c(x,y,z)\big)\\
	&= 0,
	\end{align*}
where in the {third equality} we have used that $c$ is a 3-cocycle in the Chevalley-Eilenberg complex of $\g$ with values in the representation $\h$, {and in the last equality the condition $\xi \in [\g,\h]^{\circ}$.}
\end{proof}

Hence we can consider the linear map 
\begin{align}\label{eq:Psi}
\Psi:[\g,\h]^{\circ}&\rightarrow H^{3}(\g)\\
\xi &\mapsto [d_{3}\xi]_{\g}=[\xi\circ c]_{\g}. \nonumber
\end{align}
\nbt{to the third Lie algebra cohomology group} of $\g$.

\begin{lemma}\label{prop:stariff}
	{i) Let ${\eta} \in {\h[2]}^\ast\oplus S^2(\g[1])^\ast$ 
{satisfy $\nrt{d_{CE(L)}}{\eta}=\omega_{3p}$.}	Then ${\eta}|_{{\h[2]}}$ 
 lies in the preimage of $[\omega_{3p}]_{\g}$ under $\Psi$.}
	
	{ii)
		Conversely, let   $\xi \in [\g,\h]^{\circ}$ lie in the preimage of $[\omega_{3p}]_{\g}$ under $\Psi$. Then we can find $\phi \in S^2\g[1]$ so that $\xi+\phi$ satisfies
		$\nrt{d_{CE(L)}}(\xi+\phi)=\omega_{3p}$.}
\end{lemma}

\begin{proof}
	{ If  ${\eta}$ is a solution to the system 
		$\eqref{eq:star}$, then the second equation of the system says that ${\eta}|_{{\h[2]}}\in [\g,\h]^{\circ}$ (see Rem. \ref{rem:gh}), and taking cohomology classes in the first equation
		we see that this element is mapped by $\Psi$ to $[\omega_{3p}]_{\g}$. The converse is proved reversing the argument.}
\end{proof}
{The above lemma immediately implies:
\begin{prop}\label{prop:image}
$[\omega_{3p}]_{CE(L)}=0$ if{f} $[\omega_{3p}]_{\g}$ lies in the image of $\Psi$.  
\end{prop}
}

We now give an alternative characterization of the map $\Psi$. 
Since $[\g,\h]$ is a subrepresentation of $\h$, we can look at the quotient representation $\h/{[\g,\h]}$, which is a trivial representation. Define $c_{red}:=pr\circ c:\wedge^{3}\g\to\h/[\g,\h]$. This map is a cocycle in the Lie algebra cohomology of $\g$ with values in $\h/[\g,\h]$, {as a consequence of the facts that $c$ is a cocycle and the quotient map $\h\to \h/{[\g,\h]}$ is a morphism of representations.}
Thus
 \begin{itemize}
\item the Lie algebra $\g$,
\item  the trivial $\g$-representation $\h/{[\g,\h]}$,
\item the 3-cocycle $c_{red}=pr\circ c$ 
\end{itemize}
define a minimal Lie 2-algebra
(see \S \ref{sec:momap2}). Its underlying graded vector space is  $(\h/[\g,\h])[1]\oplus \g$ and we will refer to this Lie 2-algebra as the \textit{reduced Lie 2-algebra} corresponding to $\h[1]\oplus \g$.

We can rewrite the map $\Psi$ as follows, using the reduced Lie 2-algebra:
\begin{align}\label{eq:Psired}
\Psi:(\h/[\g,\h])^{\ast}&\rightarrow H^{3}(\g)\\  \widetilde{\xi} &\mapsto [\widetilde{\xi}\circ c_{red} ]_{\g} \nonumber
\end{align}
{In other words, $\Psi$ maps $\widetilde{\xi}$ to the $\widetilde{\xi}$-component of\footnote{The isomorphism holds since $\h/[\g,\h]$ is a trivial representation of $\g$.}  
	$[c_{red}]_{\g}\in H^{3}(\g,  \h/[\g,\h])\cong H^{3}(\g)\otimes\h/[\g,\h]$.  Hence 
we can rephrase Prop. \ref{prop:image} by saying that 
$[\omega_{3p}]_{CE(L)}=0$ if{f}  $[c_{red}]_{\g}$ has a component equal to $[\omega_{3p}]_{\g}$.
}\\
 
\begin{rem}\label{rem:complred} 
	{Suppose the representation $\h$ is a completely reducible representation (this happens for instance when $\g$ is semisimple or is integrated by a compact Lie group), i.e. $\h=\oplus_{i=1}^n\h_i$ is a direct sum of irreducible sub-representations $\h_i$.
		Notice that for every $i$, either $[\g,\h_i]=\h_i$ or $\h_i$ is the trivial 1-dimensional representation. We may reorder the indices so that the 
		trivial 1-dimensional subrepresentations (if any) are exactly $\h_1,\dots,\h_k$ for some $k \le n$. Then  $\h_{red}=\oplus_{i=1}^k\h_i$ consists of these trivial sub-representations.
		Decomposing into components a $\h$-valued 3-cocycle $c\in \wedge^3 \g^*\otimes \h$ we obtain 
		a $\h_i$-valued 3-cocycle $c_i$ for every $i$, and $c_{red}$ has components $c_1,\dots,c_k$.}
	{Hence $[\omega_{3p}]_{\g}$ lies in the image of $\Psi$ if{f} some linear combination of $[c_1],\dots,[c_k]$ equals $[\omega_{3p}]_{\g}$.}
	\end{rem}

We now apply Prop. \ref{prop:image} to obtain existence statements and obstructions for $\h[1]\oplus \g$ moment maps.
\nbt{The case $[\omega_{3p}]_{\g}=0$ is not interesting, at least if $H^1(M)$ vanishes, since then a $\g$ moment map   exists (Prop. \ref{prop:9596}), and thus a $\h[1]\oplus \g$ moment map exist too (Cor. \ref{cor:gtohg}).}
Hence now we focus on the case $[\omega_{3p}]_{\g}\neq 0$.

\begin{prop}\label{prop:cred0}
	{Assume $[\omega_{3p}]_{\g}\neq 0$. If $[c_{red}]_{\g}=0\in H^{3}(\g,  \h/[\g,\h])$, then there exists no $\h[1]\oplus \g$ moment map.}
\end{prop}
\begin{proof}
	{The characterization \eqref{eq:Psired} of $\Psi$  makes clear that $[c_{red}]_{\g}=0$ if{f} $\Psi$ is the zero map. By Prop. \ref{prop:image} we have
$[\omega_{3p}]_{CE(L)}\neq 0$. We conclude using Prop. \ref{cor:exact}.}
\end{proof}

{Notice that the assumption $[c_{red}]_{\g}=0$ is implied by either of the following conditions:
	\begin{itemize}
		\item   The representation $\h$ of $\g$ satisfies $[\g,\h]=\h$, for in that case 
		$\h/[\g,\h]=0$.
		\item The cocycle $c$ satisfies\footnote{In particular, this condition is satisfied when $c=0$, i.e. $\h[1]\oplus \g$ is a graded Lie algebra. In this case, the conclusion of Prop. \ref{prop:cred0} also follows from		Rem. \ref{rem:DGLA} and Prop. \ref{prop:9596}.} $[c]=0$, because the quotient map $\h\to \h/{[\g,\h]}$ is a morphism of representations and the induced map {$H(\g,\h)\to  H(\g,\h/{[\g,\h]})$}	maps $[c]$ to $[c_{red}]$.
			\end{itemize}
} 

\begin{rem}\label{rem:ghnoth}
	{Assuming $[\omega_{3p}]_{\g}\neq 0$, if there exists a $\h[1]\oplus \g$ moment map then the representation must satisfy $[\g,\h]\neq\h$ (this follows from \nbt{the first bullet point above}).}
	{An easy  general fact about Lie algebra  representations   then implies that either $\h$ is not irreducible, or $\h$ is the trivial 1-dimensional representation. }
\end{rem}

A positive result is the following:
\begin{prop}\label{prop:1dim}
	{ Assume $[\omega_{3p}]_{\g}\neq 0$. If  $H^{3}(\g)$ is one-dimensional  and {$[c_{red}]_{\g} \neq 0$}, then 
{$[\omega_{3p}]_{CE(L)}=0$}. Thus if $H^1(M)=0$ then there exists a moment map for $\h[1]\oplus \g$.
	} 
\end{prop}
\begin{proof}
	{By the characterization \eqref{eq:Psired}, the map
		$\Psi$ is surjective. Hence the preimage of $[\omega_{3p}]_{\g}$ under $\Psi$ is nonempty. By Prop. \ref{prop:image} we have $[\omega_{3p}]_{CE(L)}=0$.
		We finish using  Thm. \ref{cor:explicit}.}
\end{proof}

\subsection{\nbt{An {alternative} characterization of existence}}\label{subsec:which}
In this subsection \nbt{we give a conceptual answer to the 
existence question for $\h[1]\oplus \g$ moment maps.} 
 
\begin{lemma}\label{lem:surjquot}
 Let $\h[1]\oplus \g$ be a minimal Lie 2-algebra. The following are equivalent:
 \begin{itemize}
\item[a)] there is a surjective\footnote{I.e. the first component is surjective.} $L_{\infty}$-morphism  from  $\h[1]\oplus \g$ to $\RR[1]\oplus_{-\omega_{3p}} \g$ which is $Id_{\g}$ on $\g$, 
\item[b)]  there is a quotient  of $\h[1]\oplus \g$ which is $L_{\infty}$-isomorphic\footnote{An $L_{\infty}$-isomorphism is an $L_{\infty}$-morphism whose first component is an isomorphism.}
 to $\RR[1]\oplus_{-\omega_{3p}} \g$
by a morphism that is $Id_{\g}$ on $\g$. 
\end{itemize}
\end{lemma}
\begin{proof}
  $a) \Leftarrow b)$:  {just compose the morphism from  $\h[1]\oplus \g$ to its quotient with the   $L_{\infty}$-isomorphism \nbt{from the latter} to $\RR[1]\oplus_{-\omega_{3p}} \g$.}
  
$a) \Rightarrow b)$:  let $f$ be a morphism as in a), and denote $\xi:= f_1|_{\h}$. Then $\ker(\xi)[1]$ is a $L_{\infty}$-ideal of $\h[1]\oplus \g$, \nbt{as can be seen from \eqref{firstsys}.}
Further, $f$ descends to an $L_{\infty}$-morphism from the quotient $(\h/\ker\nbt{(\xi)})[1]\oplus \g$ to $\RR[1]\oplus_{-\omega_{3p}} \g$.   The latter reads $Id_{\g}$ on $\g$, and is an $L_{\infty}$-isomorphism because $\xi\neq 0$ due to the surjectivity of $f$.
\end{proof}
 
 \begin{rem}
In the proof of ``$a) \Rightarrow b)$'' above, the quotient $L_{\infty}$-algebra is strictly isomorphic to $\RR[1]\oplus_{\xi \circ c} \g$ by the map $(\xi, id_{\g})$. Under this identification, the $L_{\infty}$-isomorphism given there can be alternatively described as in
  \cite[Corollary A.10]{FRZ} (notice that $[\xi \circ c]_{\g}=-[\omega_{3p}]_{\g}$ as elements of $H^3(\g)$). 
\end{rem}

{
The following statement should be compared with Prop. \ref{prop:image}.
\begin{lemma}\label{lem:op3surj}
  Assume $[\omega_{3p}]_{\g}\neq 0$. Then $[\omega_{3p}]_{CE(L)}=0$ if{f}  
 there exists a  surjective $L_{\infty}$-morphism  as in Lemma \ref{lem:surjquot} a).
  \end{lemma}
  }
\begin{proof}
{Let $\eta=\eta|_{\h[2]}+\eta|_{S^2\g[1]}\in {CE(L)}^2$ satisfy $d_{CE(L)}{\eta}=\omega_{3p}$.
 The element $\xi:= \eta|_{\h[2]}$ is non-zero, because
$[d_3\xi]_{\g}=[\omega_{3p}]_{\g}\neq0$ by Lemma \ref{prop:stariff} i).
Hence the $L_{\infty}$-morphism $f$ from $\h[1]\oplus \g$ to $\RR[1]\oplus_{-\omega_{3p}} \g$  \nbt{that corresponds to $\eta$ by} Lemma \ref{lem:solstarlimor}, which satisfies $(f_1)|_{\g}=Id_{\g}$, has surjective first component.}

For the converse, just apply  Lemma \ref{lem:solstarlimor}.\end{proof}

\nbt{The following proposition gives an {alternative} characterization of the existence of $\h[1]\oplus \g$ moment maps.}

\begin{prop} \label{prop:which2}
Assume $[\omega_{3p}]_{\g}\neq 0$.  If a $\h[1]\oplus \g$   moment map exists, then
  $\h[1]\oplus \g$ has a quotient which is $L_{\infty}$-isomorphic to
$\RR[1]\oplus_{-\omega_{3p}} \g$ {by a morphism that is $Id_{\g}$ on $\g$}. The converse holds if $H^1(M)=0$. 
\end{prop}
\begin{proof}
{Combine  Cor. \ref{prop:combined}, Lemma \ref{lem:op3surj} and Lemma \ref{lem:surjquot}.}
\end{proof}

\begin{ex}
\nbt{Suppose that $\g$ is} the Lie algebra of a simple compact Lie group. The \emph{string Lie 2-algebra} is defined as  $\RR[1]\oplus_{Car}\g$ where $Car$ is the Cartan 3-cocycle on $\g$ (see Ex. \ref{ex:comconn}), which is known to be a generator of the 1-dimensional vector space $H^3(\g)$.
Assume that ${[\omega_{3p}]_{\g}}\neq 0$. Then {$[{\omega_{3p}}]_{\g}$} and $Car$ are multiples of each other, so  $\nbt{\RR[1]\oplus_{-\omega_{3p}} \g}$ is \nbt{$L_{\infty}$-isomorphic} to the string Lie 2-algebra, \nbt{as one can see using \cite[Cor. A.10]{FRZ}}.
By Prop. \ref{prop:which2}, \nbt{if a   $\h[1]\oplus \g$ moment map exists, then necessarily $\h[1]\oplus \g$ has} the string Lie 2-algebra as a quotient. (The converse holds if $H^1(M)=0$.) 
\end{ex}

\section{Examples}\label{sec:ex}

We now  present instances in which moment maps for Lie 2-algebras exist, 
\nbt{using the explicit criteria developed in \S\ref{subsec:h3g}.}  
 
{
In this section $\g$ is always a Lie algebra, $\h$ a $\g$-representation, and $c$ 
 a 3-cocycle for this representation. \nbt{(We remind that this triple of data is equivalent to a minimal Lie 2-algebra structure on $\h[1]\oplus \g$, see \S\ref{sec:momap2}.)}
Further $(M,\omega)$ is a 2-plectic   manifold, which we assume to satisfy $$H^1(M)=0,$$ and
 $\g\to \mathfrak{X}(M),x\mapsto v_{x}$ a Lie algebra morphism  taking values in Hamiltonian vector fields. 
 
\nbt{Recall that these data give rise to a reduced Lie 2-algebra as in \S\ref{subsec:h3g}, corresponding to a   triple $(\g,\h/{[\g,\h]},c_{red})$. Further it gives rise to 
 a 3-cocyle $\omega_{3p}$ for $\g$ as in \eqref{eq:omega3p} (upon an immaterial choice of  point $p\in M$).}
  
 \begin{rem}
  If $dim (\g)\leq 2$, then there exists a $\g$ moment map  \nbt{by Prop. \ref{prop:9596} (since $\wedge^3 \g=\{0\}$), and thus also a $\h\oplus\g$ moment map by Cor. \ref{cor:gtohg}}. 
  Therefore, when looking for Lie 2-algebra moment maps \nbt{that do not arise from Lie algebra moment maps}, we should look at Lie algebras with $dim(\g)\geq3$.   
\end{rem}

\subsection{Abelian Lie algebras} 
\begin{lemma}
Suppose the Lie algebra  $\g$ is abelian.	
\nbt{Then} $[\omega_{3p}]_{CE(L)}=0$ if and only if $\omega_{3p}|_{U}=0$, where $U=\{u\in\wedge^{3}\g:\hspace{0.2cm}c(u)\in[\g,\h]\}.$
\end{lemma}

\begin{rem}\label{rem:reduced} In terms of the reduced Lie 2-algebra, the condition \nbt{$\omega_{3p}|_{U}=0$} becomes the following inclusion of kernels of linear maps:  $\ker(c_{red}\colon \wedge^3\g\to \h_{red})\subset \ker(\omega_{3p}\colon \wedge^3\g\to \RR)$. 
\end{rem}

\begin{proof}
Using Prop. \ref{prop:image} and the fact that $\g$ is abelian, we deduce that $[\omega_{3p}]_{CE(L)}=0$ iff there exists \nbt{$\xi\in [\g,\h]^{\circ}$} making this diagram commute:
		
		\[\xymatrix{   & \h \ar[d]^{\xi}
			\\
			\wedge^3 \g \ar[r]_{\omega_{3p}}\ar[ru]^{c} &\RR}
		\]
		
	Assume there is such an $\xi$. Then for $u\in \wedge^{3}\g$ such that $c(u)\in[\g,\h]$ we must have $\omega_{3p}(u)=\xi(c(u))=0$.
	
	Conversely, if  $\omega_{3p}|_{U}=0$ then we can define \nbt{$\xi|_{im(c)}$ by} $\xi|_{im(c)}(a):=\omega_{3p}(c^{-1}(a))$ \nbt{for all $a\in im(c)\subset \h$}, where $c^{-1}(a)$ is any element in the preimage of $a$ under $c$. Such $\xi|_{im(c)}$ is well-defined, because $\omega_{3p}|_{U}=0$ {implies} that $ker(c)\subset ker(\omega_{3p})$.
	By defining $\xi|_{V}=0$ on any $V$ such that $im(c)\oplus V=\h$ and extending linearly to the rest of $\h$, we obtain the desired $\xi$.
\end{proof}
Using Cor. \ref{prop:combined} we obtain:
\begin{cor}
Suppose $\g$ is an abelian Lie algebra. Then there is a $\h[1]\oplus\g$ moment map  if and only if $\omega_{3p}|_{U}=0$, where $U=\{u\in\wedge^{3}\g:\hspace{0.2cm}c(u)\in[\g,\h]\}.$
\end{cor}

\begin{ex}\label{ex:abr3}
Consider $(M,\omega)=(\RR^{3},dx\wedge dy\wedge dz)$. Let the abelian Lie algebra $\g=\RR^{3}$ act on $M=\RR^{3}$ by translations. This action preserves $\omega$, therefore, the action is generated by symplectic, hence Hamiltonian, vector fields. Using Remark \ref{rem:reduced} and noticing that $dim(\wedge^3\g)=1$, we can see:
\textit{a moment map for $\h[1]\oplus \g$ exists if and only if $c_{red}\neq 0$}.
Here are some concrete simple cases illustrating this result:

\begin{itemize}
\item Let $\h$ be any representation of $\g=\RR^3$, and $c=0$ the zero cocycle. Since 
$c_{red}=0$ in this case, we have no moment map for $(\h[1]\oplus \g,[\;,\;]_{\g},c\equiv0)$.
On the other hand, if $\h$ is a trivial representation and $c$ is any non-zero cocycle, then   $c_{red}\neq 0$, and we have a moment map for $\h[1]\oplus \g$.
\item For representations $\h$ of $\g=\RR^{3}$ such that $[\g,\h]=\h$, there is no moment map  for $\h[1]\oplus \g$, because $c_{red}=0$. This is the case, for example, for the representation
\[\hspace{0.5cm}\g\to {End(\h)}, e_{i} \mapsto \lambda_{i}Id\]
{where the $e_i$ are basis elements of $\g$, and at least one of the $\lambda_{i}\in \RR$ is non-zero.}
\end{itemize}
\end{ex}

\subsection{Other examples}

\nft{We present two examples in which the Lie algebra $\g$ is not abelian. In both of them, just as in Ex. \ref{ex:abr3},  the action is given by left translation on a Lie group with vanishing first and second cohomology.}	

\begin{ex}[Connected compact simple Lie groups]\label{ex:comconn}
As in \cite[Ex. 9.12]{FRZ},
let $G$ be a connected compact simple Lie group, acting on itself by left multiplication. Recall that $H^{1}(G)=H^{2}(G)=0$ (see, e.g., \cite{bredon2013topology}). The Lie algebra $\g$ of such a group is equipped with a skew non-degenerate trilinear form \[\theta(x,y,z):=\langle x,[y,z] \rangle\] called \emph{Cartan 3-cocycle}, where $\langle\;,\rangle$ is the Killing form.
Let $\omega$ be the  {left-invariant} 2-plectic form  on $G$ which equals $\theta$ at the identity element $e$.
\nbt{The action is Hamiltonian, and  $[\omega_{3e}]_{\g}=[\theta]_{\g}\neq 0$} in $H^{3}(\g)\cong\RR$. 

Thus, if $\h$ is any representation of $\g$ {and} $c$ a 3-cocycle for this representation,  Prop. \ref{prop:cred0} and Prop. \ref{prop:1dim} imply: 
\begin{center}
\textit{There exists a moment map for $\h[1]\oplus\g$ iff $[c_{red}]_{\g}\neq 0$}. 
\end{center}
{Notice that 
Rem. \ref{rem:complred} applies, and that using the notation introduced there, the condition $[c_{red}]_{\g}\neq 0$ can be expressed as follows: $[c_i]\neq 0 \in H^{3}(\g)$ for some $i\in \{1,\dots,k\}$.}
 \end{ex} 
 
\begin{ex}[The Heisenberg Lie algebra]
Let $\g$ be the Lie algebra of the Heisenberg group $G$, i.e.,
 \begin{equation*}
\g=\Bigg\{\begin{pmatrix}
0 & a & b\\0 & 0 & c\\ 0 & 0 & 0
\end{pmatrix}: a,b,c \in \RR \Bigg\}
\end{equation*}
\nbt{Below we will need the following claim:}
\noindent\emph{{There is a canonical isomorphism of 1-dimensional vector spaces ${\wedge}^{3}{\g}^{\ast}\cong H^{3}(\g)$}.
}

As a smooth manifold, $G\cong \RR^{3}$, hence $H^{1}(G)=0$.
Let $G$ act on itself by left multiplication, and let $\omega$ be a left-invariant volume form: thus, the generators of left translations are multi-symplectic and, since $H^{2}(G)=0$, Hamiltonian vector fields. 

Consider the representation of $G$ on $\h:=\RR^{3}$ by matrix multiplication, \nbt{and let $c$ be any 3-cocycle for this representation}. Clearly, $[\g,\h]\subsetneqq\h$,
and the quotient $\h/[\g,\h]$ is isomorphic to $\RR$. 
We have $[\omega_{3p}]_{\g}\neq 0$ for any point $p$, since   $\omega$
is a volume form and by the above claim.
 Since $H^{3}(\g)$ is 1-dimensional and $[\omega_{3p}]_{\g}\neq 0$,  {Prop. \ref{prop:cred0}, Prop. \ref{prop:1dim} and the above claim imply: 
 \begin{center}
 \textit{There exists a moment map for $\h[1]\oplus\g$ iff $c_{red}\neq 0$}.
 \end{center}
}

{To conclude, we prove the above claim. 
 Any $c\in {\wedge}^{3}{\g}^{\ast}$  is closed by dimension reasons, {yielding the surjective map
 ${\wedge}^{3}{\g}^{\ast}\to H^{3}(\g), c\mapsto [c]_{\g}$. This map is injective:}
for all $\xi\in {\wedge}^{2}{\g}^{\ast}$ we have $d_{\g}\xi=0$: using a  basis  $X,Y,Z$ of $\g$ {satisfying the bracket relations   $[X,Y]=Z,[X,Z]=[Y,Z]=0$}, we have
$$
	(\nrt{d_{\g}}\xi)(X,Y,Z)=-\xi([X,Y],Z)+\xi([X,Z],Y)-\xi([Y,Z],X)=0.
$$
	 }
\end{ex}

\appendix
\section{{Lie 2-algebra moment maps for arbitrary Lie algebra actions}}

Given an infinitesimal action of a Lie algebra $\g$ on a 2-plectic manifold, there exists a Lie 2-algebra equipped with a moment map, which we write down explicitly. However the degree zero component of this Lie 2-algebra is not $\g$ in general (it is only when $H^1(M)=0$). This is the reason why we put this statement -- which is of independent interest but is not used in the body of the paper -- in the appendix. 
 
This proposition extends our Prop. \ref{prop:chrisrg} (i.e., the $n=2$ case of  \cite[Prop. 9.10]{FRZ}) by removing the assumption that $H^1(M)=0$.
\begin{prop}\label{prop:gH}
Let $\g$ be a  Lie  algebra, let $\omega$ be a 2-plectic form on a manifold $M$, and let $\g\to \mathfrak{X}(M), x\mapsto v_{x}$ be a Lie algebra morphism taking values in Hamiltonian vector fields. 
\begin{itemize}
\item[i)] These data determine a Lie algebra structure on $\g\oplus H^1(M)$, which is a central extension of $\g$ and which is  canonical up to isomorphism.
\item[ii)]
There exists a moment map for the Lie 2-algebra $$\RR[1]\oplus_{-\omega_{3p}} (\g\oplus H^1(M)).$$ This Lie 2-algebra is defined precisely as in \S\ref{subsec:imme}, regarding  $-\omega_{3p}$   as a 3-cocycle for $\g\oplus H^1(M)$.
\end{itemize}
\end{prop}
Notice that the degree zero component of the above Lie 2-algebra is $\g$ only when $H^1(M)=0$.
Hence, when $H^1(M)\neq 0$, Prop. \ref{prop:gH} does not address the question of existence of a moment map for a Lie 2-algebra which has $\g$ in  degree zero.
\begin{proof}
i) Denote by $\alpha\colon \g\to {\Omega}_\textrm{Ham}^{1}(M)$  any linear map such that $\iota_{v_x}\omega=-d\alpha_x$ for all $x\in \g$. One checks easily that $\alpha_{[x,y]}-\iota(v_x\wedge v_y)\omega$ is closed for all $x,y\in \g$.
We define $$B\colon \wedge^2\g\to H^1(M),\;\; B(x,y)= \left[\alpha_{[x,y]}-\iota(v_x\wedge v_y)\omega\right]_{H^1(M)},$$
where $[\,\cdot\,]_{H^1(M)}$ denotes the de Rham cohomology class. It turns out that $B$ is a Lie algebra 2-cocycle for $\g$ with values in the trivial representation $H^1(M)$: to see this, use the Jacobi identity, and the fact that for all
$x_1,x_2,x_3\in \g$ 
(see for instance \cite[Lemma 9.2]{FRZ})
\begin{equation}\label{eq:9.2}
\iota(v_{[x_1,x_2]}\wedge v_{x_3}) \omega-\iota(v_{[x_1,x_3]}\wedge v_{x_2}) \omega +\iota(v_{[x_2,x_3]}\wedge v_{x_1}) \omega=d(\iota(v_{x_1}\wedge v_{x_2}\wedge v_{x_3})\omega).
\end{equation}
 Notice that while $B$ depends on the choice of $\alpha$, its class  in the Chavalley-Eilenberg cohomology of $\g$ with values on $H^1(M)$ does not: if $\alpha'$ is another linear map as above, we have $B_{\alpha}-B_{\alpha'}=-d_{\g}[\alpha-\alpha']_{H^1(M)}$.

We consider the central extension of $\g$ by the 2-cocycle $B$, i.e. the Lie algebra $\g\oplus H^1(M)$ with bracket $$[(x_1,A_1),(x_2,A_2)]=([x_1,x_2],B(x_1,x_2)).$$ As a consequence of the above comment, the isomorphism class of the extension does not depend on the choice of $\alpha$.

ii) This part of the proof extends the one of \cite[Prop. 9.10]{FRZ}. Fix a linear map $\psi\colon H^1(M)\to \Omega^1_{closed}(M)$ which assigns a representative to each cohomology class. 
 We claim that the following are the components of an $L_{\infty}$-morphism from $\RR[1]\oplus_{-\omega_{3p}} (\g\oplus H^1(M))$ to $L_{\infty}(M,\omega)$:
\begin{align*}
F_1|_{\g\oplus H^1(M)}&:\g\oplus H^1(M)\to {\Omega}_\textrm{Ham}^{1}(M),\;\; x+A\mapsto \alpha_x-\psi(A)\\
F_1|_{\RR[1]}&:\RR\to C^{\infty}(M), \;\;r\mapsto r \\
F_2&:{\wedge}^2(\g\oplus H^1(M))\to C^{\infty}(M).
\end{align*}
Here we define $F_2(x_1+A_1,x_2+A_2)$ to be the only function vanishing at $p$ which is a primitive of the exact 1-form 
$$\alpha_{[x_1,x_2]}-\iota(v_{x_1}\wedge v_{x_2})\omega-\psi(B(x_1,x_2)).$$
Notice that   $F_2(x_1+A_1,x_2+A_2)$ has no dependence on $A_1$ or $A_2$,
and that the hamiltonian vector field of $F_1(x+A)$ is $v_x$.
The $L_{\infty}$-morphism relations read as follows (see for example \cite[Prop. A.9]{FRZ}), for all $r\in \RR, x_i\in \g, A_i\in H^1(M)$:
\begin{align}
\label{eq:H1}&d(F_1(r))=0\\
\label{eq:H2}&F_1[x_1, x_2]=d(F_2(x_1,x_2))+\iota(v_{x_1}\wedge v_{x_2})\omega\\
\label{eq:H3}&F_2([x_1, x_2],x_3)+c.p.=\omega_{3p}(x_1,x_2,x_3)-\iota(v_{x_1}\wedge v_{x_2}\wedge v_{x_3})\omega,
\end{align}
where $\omega_{3p}$ was defined in eq. \eqref{eq:omega3p}.
Eq. \eqref{eq:H1} is trivially satisfied, and eq. \eqref{eq:H2} is satisfied due to the way $F_2$ was defined. To tackle eq. \eqref{eq:H3}, for all $x_1,x_2,x_3\in \g$, define the following function on $M$:
$$C(x_1,x_2,x_3)=F_2(x_1,[x_2,x_3])+c.p.+(\iota(v_{x_1}\wedge v_{x_2}\wedge v_{x_3})\omega)|_p-\iota(v_{x_1}\wedge v_{x_2}\wedge v_{x_3})\omega.$$
To shorten the notation, let $X$ denote  the l.h.s. of eq. \eqref{eq:9.2}. Using that equation, one computes that the differential of the function $F_2(x_1,[x_2,x_3])+c.p.$ is $X-\psi[X]_{H^1(M)}=X$ and that therefore $C(x_1,x_2,x_3)$ is a constant function. Since $C(x_1,x_2,x_3)$ vanishes at $p$, it must vanish identically, showing that eq. \eqref{eq:H3} holds too.
\end{proof}

\begin{rem} As the referee    pointed out to us, in the set-up of Prop. \ref{prop:gH}, 
\cite[\S3.4]{FiorenzaRogersSchrLocObs} constructs a Lie 2-algebra $\mathfrak{h}\mathfrak{e}\mathfrak{i}\mathfrak{s}_{\rho}(\g)$ admitting a moment map, using methods from homotopy theory. (We already recalled this fact in Rem. \ref{rem:heis}.)
 We obtained the statement of Prop. \ref{prop:gH} ii) trying to describe  $\mathfrak{h}\mathfrak{e}\mathfrak{i}\mathfrak{s}_{\rho}(\g)$ explicitly. To this aim, 
 we used  the fact that the cochain complex $C^{\infty}(M)\to \Omega^1(M)\to d\Omega^1(M)$ (concentrated in degrees $-2,-1,0$) is quasi-isomorphic to
its cohomology $\RR[2]\oplus H^1(M)[1]$.
\end{rem}

\bibliographystyle{habbrv} 

\end{document}